\documentclass[11pt]{article}
\usepackage{amsmath, amsthm, amssymb, amsfonts}
\usepackage[utf8]{inputenc}
\usepackage{enumitem}
\usepackage[pdftex, colorlinks, linkcolor=vdarkblue,citecolor=vdarkred]{hyperref}
\usepackage{amsmath}
\usepackage{amssymb}
\usepackage{amsfonts}
\usepackage{latexsym}
\usepackage{amsthm}
\usepackage{amsxtra}
\usepackage{amscd}
\usepackage{bbm}
\usepackage{mathrsfs}
\usepackage{bm}
\usepackage{tensor}
\usepackage{comment}
\usepackage{mathtools}
\usepackage{geometry}
\usepackage{graphicx, color}
\numberwithin{equation}{section}

\theoremstyle{plain} 
\newtheorem{theorem}{Theorem}[section]
\newtheorem*{theorem*}{Theorem}
\newtheorem{lemma}[theorem]{Lemma}
\newtheorem*{lemma*}{Lemma}
\newtheorem{corollary}[theorem]{Corollary}
\newtheorem*{corollary*}{Corollary}
\newtheorem{proposition}[theorem]{Proposition}
\newtheorem*{proposition*}{Proposition}

\newtheorem*{conjecture*}{Conjecture}

\newtheorem*{question*}{Question}

\theoremstyle{definition} 
\newtheorem{definition}[theorem]{Definition}
\newtheorem*{definition*}{Definition}

\newtheorem*{example*}{Example}
\newtheorem{remark}[theorem]{Remark}
\newtheorem*{remark*}{Remark}

\newtheorem*{assumption*}{Assumption}

\newcommand*{\deq}{\mathrel{\vcenter{\baselineskip0.65ex \lineskiplimit0pt \hbox{.}\hbox{.}}}=}
\newcommand{\bb}{\mathbb} 
\renewcommand{\cal}{\mathcal} 
\newcommand{\E}{\mathbb{E}}
\newcommand*{\eqd}{=\mathrel{\vcenter{\baselineskip0.65ex \lineskiplimit0pt \hbox{.}\hbox{.}}}}

\newcommand{\R}{\mathbb{R}}
\newcommand{\e}{\mathrm{e}}
\newcommand{\C}{\mathbb{C}}
\newcommand{\N}{\mathbb{N}}

\newcommand{\Tr}{\operatorname{Tr}}

\newcommand{\dd}{\mathrm{d}}

\newcommand{\numberthis}{\addtocounter{equation}{1}\tag{\theequation}}

\newcommand{\txt}[1]{\text{\rm{#1}}}
\newcommand{\col}{\mathrel{\vcenter{\baselineskip 0.75ex \lineskiplimit 0pt \hbox{.}\hbox{.}}}}
\renewcommand{\epsilon}{\varepsilon}
\allowdisplaybreaks

\definecolor{vdarkred}{rgb}{0.7,0,0.2}
\definecolor{vdarkblue}{rgb}{0,0.2,0.7}

\title{Fluctuations of functions of sparse Erd\H{o}s-R\'enyi graphs}

 \author{Hok-Yin Chu\footnote{University of Oxford, Mathematical Institute. Email: {\tt hok.chu@merton.ox.ac.uk}.}} 

\date{July 19, 2025}

\begin{document}

\maketitle

\begin{abstract}
 Let $A$ be the (rescaled) adjacency matrix of the Erd\H{o}s-R\'enyi graphs $\cal G(N,p)$. For $N^{-1+\tau} \leqslant p\leqslant N^{-\tau}$, we study the fluctuation of $f(A)_{ii}$ on the global and mesoscopic spectral scales. We show that the distribution of $f(A)_{ii}$ is asymptotically the sum of two independent Gaussian random variables on different scales, where a phase transition occurs on the spectral scale $p$.  
\end{abstract}
\section{Introduction}
Fix small $\tau>0$. In this paper, we consider the following class of random matrices.  
\begin{definition} [Sparse matrix] \label{def:sparse} Let $q\in [N^{\tau/2},N^{1/2-\tau/2}]$. Consider a real-symmetric $N\times N$ matrix $H$ whose entries $H_{ij}$ satisfy the following conditions.
	\begin{enumerate}
		\item[(i)] The upper-triangular entries ($H_{ij}\col 1 \leqslant i \leqslant j\leqslant N$) are independent.
		\item[(ii)] We have $\bb E H_{ij}=0$, $ \bb E H_{ij}^2=(1+O(\delta_{ij}))/N$, and $\bb E H_{ij}^4\asymp 1/(Nq^2)$ for all $i,j$.
		\item[(iii)] For any $k\geqslant 3$, we have
		$\bb E|H_{ij}|^k \leqslant C_k/ (Nq^{k-2})$ for all $i,j$.
	\end{enumerate}
	We define the random matrix
	$$
	A = H + f \e \e^*\,,
	$$
	where $\e \deq N^{-1/2} (1,1,\dots,1)^*$, and $f \asymp q$.
\end{definition}
One major motivation for Definition \ref{def:sparse} is the sparse Erd\H{o}s-R\'enyi graph $\cal G(N,p)$. More precisely, it is an undirected graph on $N$ vertices, and each edge is connected with probability $p$, independent from any other edges. Let $\cal A$ denote the adjacency matrix of the graph. Explicitly, $\cal A = (\cal A_{ij})_{i,j = 1}^N$ is a symmetric $N\times N$ matrix with independent upper triangular entries $(\cal A_{ij} \col i \leqslant j)$ satisfying
\begin{equation*}
	{\cal A}_{ij}=\begin{cases}
		1 & \txt{with probability } p
		\\
		0 & \txt{with probability } 1-p\,.
	\end{cases}
\end{equation*} 
We introduce the normalized adjacency matrix
\begin{equation} \label{1.11}
	A\deq \sqrt{\frac{1}{p(1-p)N}}\,\cal A\,,
\end{equation}
where the normalization is chosen so that the eigenvalues of $A$ are typically of order one. More precisely, let $\lambda_1\geqslant \cdots \geqslant \lambda_N$ be the eigenvalues of $A$. It can be shown that the empirical eigenvalue density of $A$ satisfies
\begin{equation}
	\mu(x)\deq \frac{1}{N}\sum_{i=1}^N\delta_{\lambda_i}(x) \overset{w}{\longrightarrow}\varrho_{sc}(x)\deq \frac{1}{2\pi} \sqrt{(4-x^2)_+}
\end{equation}
almost surely as $N\to \infty$.  It is easy to check that  when $N^{-1+\tau} \leqslant p\leqslant N^{-\tau}$, the rescaled adjacency matrix $A$ satisfies Definition \ref{def:sparse} with $q\deq\sqrt{Np}$.

The matrix $A$ has typically $N^2p$ nonzero entries, and hence $A$ is \textit{sparse} whenever $p\to 0$ as $N\to \infty$. The celebrated Wigner-dyson-Mehta (WDM) universality conjecture asserts that the local spectral properties of a random matrix do not depend on the explicit distribution of the matrix entries, and they are only determined by the symmetry class of the matrix. During the past decade, the
universality conjecture for sparse matrices has been established in a series of papers \cite{EKYY1,EKYY2,HLY15,LS1,HLY,HK20,Lee21,HY22} in
great generality. More precisely, it has been shown that when $p\geqslant N^{-1+o(1)}$, the averaged
$n$-point correlation functions and the distribution of a single eigenvalue gap of $A$ coincide with those of the
GOE, and the edge eigenvalues of $A$ have Tracy-Widom distributions.

Another important topic in random matrix theory is the study of linear eigenvalue statistics $\Tr f(A)$. When the graph is dense, $H\deq A-\bb EA$ is essentially a Wigner matrix, and the distribution of $\Tr f(H)$ was obtained both on the global \cite{LP,BWZ09} and mesoscopic scales \cite{HK}. For sparse matrices, the distribution of $\Tr f(H)$ was computed in \cite{ST12,H19}, where \cite{ST12} treated the global scale, and \cite{H19} handled the mesoscopic scales with the special test function $f(x)=(x-\mathrm{i})^{-1}$.

As a natural extension of the linear statistics, one can also study the fluctuations of functions (i.e. $f_{ij}(H)\deq f(H)_{ij}$) of random matrices. For Wigner matrices, the distribution of $f_{ij}(H)$ was derived both on the global \cite{ORS,ES16} and mesoscopic \cite{CEK6} scales. 

In this paper, we study the fluctuation of $f_{ii}(A)$ on the sparse levels $N^{-1+\tau} \leqslant p\leqslant N^{-\tau}$. Our test functions $f\equiv f_N \in C^{\infty}(\bb R)$ live on the scale $\eta_*\in [N^{-1+\tau},1]$. More precisely, let $F\in C_c^{\infty}(\bb R)$ be a function independent of $N$ and $E\in[-2+\tau, 2-\tau]$. Then \[f(x)\deq F\biggl(\frac{x-E}{\eta_*}\biggr)\,.\]
We further require that
\[f'(x) \neq0 \text{ only if }x\in (-2+\tau, 2-\tau)\,.\]

Assuming that all off-diagonal entries of $H$ are identically distributed, we may now state our main result.
\begin{theorem}[Convergence of general test functions on $A$] \label{thm:1} Let $f$ be as above and $N_{\bb R}(0,1)$ be a standard Gaussian random variable. Moreover, let $S$ denote the random variable with density $\varrho_{sc}(x)\deq \frac{1}{2\pi}\sqrt{(4-x^2)_+}$. Then
\[\frac{f_{ii}(A)-\bb E[f_{ii}(A)]}{\sqrt{V_{ii}(f)}}\xrightarrow{d} N_{\bb R}(0,1)\] 
as $N \to \infty$, where $V_{ii}(f)$ is defined as
\begin{align*}
&\,
\frac{2}{N}\biggl(\int_{-2}^2 f(x)^2\varrho_{sc}(x)\mathrm{d}x-\biggl(\int_{-2}^2 f(x)\varrho_{sc}(x)\mathrm{d}x\biggr)^2\biggr)+N \mathcal{C}_{4}(H_{12})\biggl(\int_{-2}^2 f(x)(1-x^2)\varrho_{sc}(x)\mathrm{d}x \biggr)^2
\\
 =&\,\frac{2}{N}\mathrm{Var}(f(S)) +N\cal C_4(H_{12}) \big(\bb E [f(S)(1-S^2)]\big)^2
\end{align*}
and
\[\bb E [f_{ii}(A)] = \int_{-2}^2 f(x)\varrho_{sc}(x)\mathrm{d}x + O(N^{-\tau/100}V_{ii}(f)^{1/2})\,.\]
\end{theorem}
Here, $\cal C_4$ denotes the fourth cumulant.
\begin{remark}
	(i) The first term on $V_{ii}(f)$ is of order $\eta_*/N$, while the second term is of order $\eta^2_*/q^2$. Thus, for the diagonal entries $f_{ii}(A)$, there is a phase transition of the fluctuation on the scale $\eta_*/N\asymp \eta^2_*/q^2$, i.e.\,$\eta_*\asymp q^2/N= p$.

(ii) For simplicity, here we only state the result for $f_{ii}(A)$. It can be checked that the same result also holds for $f_{ii}(H)$.
\end{remark}

Our proof begins with the strategies of \cite{HK,H,LS18}, by converting the general test function to the Green function, and then compute through the cumulant expansion formula. Notably, for sparse matrices, we do not need to remove the diagonal contribution $H_{ii}\int f(x) x \varrho_{sc}(x)\dd x$ as in the Wigner case \cite{ES16}, since this term is always negligible in the sparse regime $p\leqslant N^{-o(1)}$. In addition, we have new terms arising from the large expectation of $A$. 

\paragraph{Acknowledgment}The author is partially supported by Hong Kong RGC Grant
No.\,21300223.
\section{Preliminaries}
Define
\[G\deq (A-z)^{-1}\,.\]
Let $W$ be an $N\times N$ matrix. We denote the normalized trace of $W$ by $\underline{W}\deq \frac{1}{N} \operatorname{Tr}W$ and define $\langle X \rangle \deq X- \bb EX$ for any random variable $X$ with finite expectation. For the Green function $G$, we have the following differential rule.
\begin{equation}
\label{eq:H_diff}
\frac{\partial G_{ij}}{\partial H_{kl}} =-(G_{ik}G_{jl}+G_{il}G_{kj})(1+\delta_{kl})^{-1}\,.
\end{equation}
The Stieltjes transform of $\varrho_{sc}(x)$ is defined by
\[\qquad m(z) \deq \int\frac{\varrho_{sc}(x)}{x-z}\,\mathrm{d}x\,.\]
It can be verified that $m(z)$ satisfies the equation
\begin{equation}
\label{eq:m}
m(z)^2+zm(z)+1=0\,.
\end{equation}
For a real-valued random variable $h$ with finite moments of all orders, the n-th cumulant is defined as 
\[\mathcal{C}_{n}(h) \deq (-1)^{n}\biggl(\frac{d^n}{dt^n} \mathrm{log} \bb E[e^{\mathrm{i}th}]\biggl)\bigg|_{t=0}\,.\]
Our main tool of computation is the cumulant expansion formula, whose proof is given in e.g. \cite[Appendix A]{HKR}.
\begin{lemma}[Cumulant expansion]\label{lemma:cumulant}
Let $f\col\R\to\C$ be a smooth function, and denote by $f^{(r)}$ its $r$-th derivative. Then, for every fixed $\ell \in\N$, we have 
\begin{equation}\label{eq:cumulant_expansion}
\mathbb{E}\big[h\cdot f(h)\big]=\sum_{r=0}^{\ell}\frac{1}{r!}\mathcal{C}_{r+1}(h)\mathbb{E}[f^{(r)}(h)]+\cal R_{\ell+1}\,,
	\end{equation}	
	assuming that all expectations in \eqref{eq:cumulant_expansion} exist, where $\cal R_{\ell+1}$ is a remainder term (depending on $f$ and $h$), such that for any $t>0$,
	\begin{equation} \label{remainder}
		\cal R_{\ell+1} = O(1) \cdot \bigg(\E\sup_{|x| \le |h|} \big|f^{(\ell+1)}(x)\big|^2 \cdot \E \,\big| h^{2\ell+4} \mathbf{1}_{|h|>t} \big| \bigg)^{1/2} +O(1) \cdot \bb E |h|^{\ell+2} \cdot  \sup_{|x| \le t}\big|f^{(\ell+1)}(x)\big|\,.
	\end{equation}
\end{lemma}

By Definition \ref{def:sparse} and Jensen's inequality, we can bound the cumulants of $A$.
\begin{lemma} For every $r \in \bb N$, we have
\[\mathcal{C}_{r}(A_{ij})=O_{r}(1/Nq^{r-2})\]
uniformly for all $i,j$.
\end{lemma}
The following tool from functional analysis will be useful.
\begin{lemma}[Helffer-Sjöstrand formula \cite{Davies}]\label{lemma:HS}
Let \(f\in C^2(\mathbb{R})\). Let \(\tilde{f}\) be the almost analytic extension of \(f\) defined by
\[\tilde{f}(x+{\rm i}y)\deq f(x)+{\rm i}yf'(x)\,.\]
Let \(\chi\in\mathcal{C}_{c}^{\infty}(\mathbb{R})\) be a cutoff function satisfying \(\chi(0)=1\), and by a slight abuse of notation write \(\chi(z)\equiv\chi(\operatorname{Im}z)\). Then for any \(\lambda\in\mathbb{R}\), we have
\[f(\lambda)=\frac{1}{\pi}\int_{\mathbb{C}}\frac{\partial_{\bar{z}}(\tilde{f}(z)\chi(z))}{\lambda-z}\,{\rm d}^{2}z\,,\]
where \(\partial_{\bar{z}}\deq \frac{1}{2}(\partial_{x}+{\rm i}\partial_{y})\) is the antiholomorphic derivative and \({\rm d}^{2}z\) the Lebesgue measure on \(\mathbb{C}\).
\end{lemma}
\begin{definition}[Stochastic domination]
Let
\[X=\biggl(X^{(N)}(u):N\in\mathbb{N},u\in U^{(N)}\biggr)\,,\qquad Y=\biggl(Y^{(N)}(u):N\in\mathbb{N},u\in U^{(N)}\biggr)\]
be two families of nonnegative random variables, where \(U^{(N)}\) is a possibly \(N\)-dependent parameter set. We say that \(X\) is stochastically dominated by \(Y\), uniformly in \(u\), if for all (small) \(\varepsilon>0\) and (large) \(D>0\) we have
\[\sup_{u\in U^{(N)}}\mathbb{P}\biggl[X^{(N)}(u)>N^{\varepsilon}Y^{(N)}(u)\biggr] \leqslant N^{-D}\]
for large enough \(N\geqslant N_{0}(\varepsilon,D)\). If \(X\) is stochastically dominated by \(Y\), we use the notation \(X\prec Y\), or equivalently $X=O_{\prec}(Y)$.
\end{definition}
We recall the local semicircle law for $A$ from \cite[Theorem 2.9]{EKYY1}. Define the domain
\[\mathbf{S}\deq \{E+\mathrm{i}\eta:|E|\leqslant 4,0<\eta\leqslant 4\}\,.\]
\begin{theorem}[Local semicircle law for A]\label{thm:semiA}
We have the bounds
\[\max_{i,j}|G_{ij}(z)-\delta_{ij}m(z)| \prec \frac{1}{q}+\sqrt{\frac{\operatorname{Im}m(z)}{N\eta}}+\frac{1}{N\eta}\]
and
\[|\underline{G}(z)-m(z)| \prec \frac{1}{q} \wedge \frac{1}{q^2(\eta + |2-|E||)}+\frac{1}{N\eta}\,,\]
uniformly in \(z=E+\mathrm{i}\eta\in\mathbf{S}\).
\end{theorem}
\begin{corollary}\label{corollary:diff}
    \[\bigg|\frac{\partial^r G_{ij}}{\partial z^r}(z) - \delta_{ij}m^{(r)}(z)\bigg| \prec \frac{1}{\eta^{r}}\biggl(\frac{1}{q} + \frac{1}{\sqrt{N\eta}}\biggr)\]
uniformly in \(z=E+\mathrm{i}\eta\in\mathbf{S}\).
\end{corollary}
\begin{proof}
    For any $z=E+\rm{i}\eta\in\mathbf{S}$, let $\partial B$ be a circle of radius $\eta/2$ centered at $z$. By the Cauchy's integral formula, assuming we could apply stochastic domination inside the integral (See \cite[Remark 2.7 and Lemma 10.2]{BK16} for further details), then
    \begin{align*}
    \bigg|\frac{\partial^r G_{ij}}{\partial z^r}(z) - \delta_{ij}m^{(r)}(z)\bigg| &= \frac{1}{2\pi \rm{i}}\oint_{\partial B} \bigg|\frac{G_{ij}-\delta_{ij}m}{(z'-z)^{r+1}}\bigg|\mathrm{d}z'\\
    &\prec \frac{1}{2\pi}\oint_{\partial B} \bigg|\frac{q^{-1}+(Ny')^{-1/2}}{(z'-z)^{r+1}}\bigg|\mathrm{d}z'\prec \frac{1}{\eta^{r}}\biggl(\frac{1}{q} + \frac{1}{\sqrt{N\eta}}\biggr)\,. 
    \end{align*}
\end{proof}
We also recall the isotropic local law from \cite[Theorem 1.4]{HCJ}. Define the domain
\[
\mathbf{\widetilde{S}}\deq \{E+\mathrm{i}\eta:|E|\leqslant 4,N^{-1+\tau/100}\leqslant \eta\leqslant 4\}.\]
\begin{theorem}[Isotropic local law for A]
\label{thm:isoA}
    \[\bigg|\frac{1}{\sqrt{N}}\sum_{k=1}^NG_{ik}\bigg| \prec q^{-1}\]
uniformly in $z=E+\mathrm{i}\eta \in\mathbf{\widetilde{S}}$.
\end{theorem}
\section{Proof of Theorem \ref{thm:1}} \label{section:4}
In this section, we denote $z=x+\mathrm{i}y$ unless otherwise specified. Define the region
\[\mathbf{D} \deq \{x+\mathrm{i}y: |4-x^2|+|y| \geqslant \tau/100, |E| \leqslant 4, N^{-1+\tau/100}\leqslant y\leqslant 4\}\,.\]
\begin{lemma}
    \begin{equation} \label{lemma:cumu}
	\bb E G_{ii}-m \prec N^{-\tau/100}\cdot (1/q + 1/\sqrt{Ny}) \eqd \mathcal{T}
\end{equation}
uniformly in $z=x+\mathrm{i}y \in \mathbf{D}$.
\end{lemma}
Define, for $\lambda \in \bb R$,
\begin{align*}
e(\lambda) &\deq \text{exp} \{ \mathrm{i} \lambda(f_{ii}(A)-\bb E[f_{ii}(A)]) /\sqrt{V_{ii}(f)}\}\,,\\
    \psi_{ii}(\lambda)&\deq\bb E[e(\lambda)]\,.
\end{align*}
\begin{proposition}\label{prop:char}
\begin{equation}
\label{eq:char}\psi_{ii}'(\lambda)
=-\lambda \psi_{ii}(\lambda)\widetilde{V}_{ii}(f)/V_{ii}(f)
+ O_{\prec}(N^{-\tau/200})\,,
\end{equation} 
where $\widetilde{V}_{ii}(f)$ is defined as
\[
\frac{1}{\pi^2}\int_{\bb C^2}\partial_{\bar{z}}\tilde{f}(z)\partial_{\bar{z'}}\tilde{f}(z')\biggl[2N^{-1} m(z)m(z')\frac{m(z')-m(z)}{z'-z} + N\mathcal {C}_{4}(H_{12})m^3(z)m^3(z')\biggr]\mathrm{d}^2z'\mathrm{d}^2z\,.
\]
\end{proposition}
\subsection{Proof of Lemma \ref{lemma:cumu}}
Note that $z+\bb E \underline{G}$ is bounded, so we have
\begin{align*}
|\bb EG_{ii}-m|&=O ( (z+\bb E\underline{G})(\bb EG_{ii}-m))\\&=O(z\bb EG_{ii}+\bb E\underline{G}\bb EG_{ii}+1+m(m-\bb E\underline{G}))\\
&=O(\bb E (GA)_{ii}+\bb E\underline{G}\bb EG_{ii})+O_\prec(\mathcal T)\\&= O(\bb E (GA)_{ii}+\bb E\underline{G}G_{ii})+O_\prec(\mathcal T)\,,\numberthis \label{eq:self}
\end{align*}
where \cite[Proposition 3.1]{H19} is applied for the last step, which says
\[ 
\underline{G}-\bb E \underline{G} \prec \frac{1}{Ny}+\frac{1}{\sqrt{N}q}\prec \cal T\
\] uniformly in $z\in \mathbf{D}$. Next, we perform a cumulant expansion on $\bb E (GA)_{ii}$. By Lemma \ref{lemma:cumulant},
\begin{align*}
\bb E (GA)_{ii}&=\sum_{k=1}^N \bb E G_{ik}H_{ki} + \sum_{k=1}^{N}\bb EG_{ik}f(ee^*)_{ki}\\&=N^{-1}\sum_{k=1}^N\bb E\frac{\partial G_{ik}}{\partial H_{ki}}(1+\delta_{ki})  -N^{-1} \bb E\frac{\partial G_{ii}}{\partial H_{ii}} +\sum_{r=2}^l \widetilde{L}_{r} + \sum_{k=1}^N \widetilde{R}_{l+1}^{(k)} + \sum_{k=1}^{N}\bb EG_{ik}f(ee^*)_{ki}\\
&\deq (\tilde{a}) -N^{-1} \bb E\frac{\partial G_{ii}}{\partial H_{ii}}+\sum_{r=2}^l\widetilde{L}_{r} + \sum_{k=1}^N \widetilde{R}_{l+1}^{(k)}+(\tilde{c})\,, \numberthis \label{eq:cumulantG}
\end{align*}
where
 \[ \widetilde{L}_r = \sum_{k=1}^N \frac{1}{r!} \mathcal{C}_{r+1}(H_{ki}) \bb E \left(\frac{\partial^r G_{ik}}{\partial H_{ki}^r} \right)\,.
\]
By Theorem \ref{thm:isoA} and the fact that $f\asymp q$, we have
\begin{align*}
        (\tilde{c}) = \sum_{k=1}^N\bb E[G_{ik}f(ee^*)_{ki}] &= \frac{f}{N}\bb E\biggl[\sum_{k=1}^N G_{ik}\biggr]\\
        &= O_{\prec}(N^{-1/2})\\
        &= O_{\prec}(\mathcal{T})\,. \numberthis \label{eq:cumulantGc}
        \end{align*}
By the differential rule \eqref{eq:H_diff}, Theorem \ref{thm:semiA} and Corollary \ref{corollary:diff}
\begin{align*}
    (\tilde{a}) &= N^{-1}\sum_{k=1}^N \bb E(-G_{ii}G_{kk}-G_{ik}G_{ki}) \\
    &= -\bb E(\underline{G}G_{ii})-N^{-1}\left[ \bb E(G^2)_{ii} -m'+m'\right]\\
    &= -\bb E(\underline{G}G_{ii}) +O_{\prec}\left( (Ny)^{-1}\left(\frac{1}{q} + \frac{1}{\sqrt{Ny}}\right) +N^{-1}y^{-1/2}\right) \\
    &= -\bb E(\underline{G}G_{ii}) +O_{\prec}(\mathcal{T})\,, \numberthis \label{eq:cumulantGa}
\end{align*}
where we used the fact that $|m'(z)| \prec y^{-1/2}$. Also, using \eqref{eq:H_diff},
\begin{equation} \label{eq:K}
    N^{-1} \bb E\frac{\partial G_{ii}}{\partial H_{ii}} = -N^{-1}G_{ii}G_{ii} \prec\mathcal{T}\,.
\end{equation}
Substituting \eqref{eq:cumulantGc}-\eqref{eq:K} into \eqref{eq:cumulantG}, and letting $l$ be large enough such that the remainder term is negligible, we have
\begin{equation} \label{0.15}
\begin{aligned}
\bb E G_{ii}-m &\prec \sum_{r=2}^{l}\sum_{k=1}^{N} \frac{1}{Nq^{r-1}}\bigg|\bb E\frac{\partial^r G_{ik}}{\partial H_{ki}^r}\bigg| +\mathcal{T}\\
&\prec \sum_{r=2}^{l}\max_{k:k\neq i} \frac{1}{q^{r-1}}\bigg|\bb E\frac{\partial^r G_{ik}}{\partial H_{ki}^r}\bigg| +\mathcal{T}\,.
	\end{aligned}
\end{equation}
Note that \eqref{0.15}, combined with Theorem \ref{thm:semiA} implies $\bb EG_{ii}-m\prec \mathcal{T}$.
\subsection{Proof of Proposition \ref{prop:char}}
By Lemma \ref{lemma:HS},
    \[f_{ii}(A) = \frac{1}{\pi}\int_{\bb C} \biggl(\frac{\partial_{\bar{z}}\tilde{f}(z)}{A-z}\biggr)_{ii} \mathrm{d}^2z\,,\]
  where $\tilde{f}(x+{\rm i}y) \deq \chi(y/\eta_*)(f(x)+{\rm i}yf'(x))$ is a quasi-analytic extension of f and $\chi(y)$ is a smooth cut-off function that is $1$ for $|y|\leq1$ and $0$ for $|y| \geqslant 2$. Fix small $\alpha > 0$ and define the domain
\[
    \Omega_{\alpha} \deq \{ (x,y) \in \bb R^2 : |y| > N^{\alpha-1}\}\,.
\]
Theorem \ref{thm:semiA} implies $|yG_{ii}| = O_{\prec}(|ym| + |y|q^{-1} + |y|^{1/2}N^{-1/2})$. Therefore, we have
\begin{align*}
 &\quad f_{ii}(A) - \bb E[f_{ii}(A)] 
  \\ &=\frac{1}{2\pi}\int_{\bb R^2} (\mathrm{i}y\chi(y/\eta_*)f''(x) + \mathrm{i}\eta^{-1}_*(f(x)+\mathrm{i}f'(x)y)\chi'(y/\eta_*))(G_{ii}-\bb EG_{ii})  \mathrm{d}x\mathrm{d}y \\
         &=\frac{1}{2\pi}\int_{\Omega_{\alpha}} (\mathrm{i}y\chi(y/\eta_*)f''(x) + \mathrm{i}\eta^{-1}_*(f(x)+\mathrm{i}f'(x)y)\chi'(y/\eta_*))(G_{ii}-\bb EG_{ii})  \mathrm{d}x\mathrm{d}y + O_{\prec}(N^{2\alpha -2} \|f''\|_{1}) \,.
     \numberthis \label{eq:e}
\end{align*}
Since $\bar{G}_{ii}(z) = G_{ii}(\bar{z})$, the error term in \eqref{eq:e} is real-valued, hence the integral over $\Omega_{\alpha}$ is also real-valued. Note also that $|e(\lambda)| \leq1$ as $f$ is a real-valued function, hence
\begin{align*}
     &\quad \psi_{ii}'(\lambda) \\&=\bb E[\mathrm{i}(f_{ii}(A) - \bb E[f_{ii}(A)])e(\lambda)/\sqrt{V_{ii}(f)}] \\
     &= \frac{\mathrm{i}}{2\pi \sqrt{V_{ii}(f)}}\int_{\Omega_{\alpha}} (\mathrm{i}y\chi(y/\eta_*)f''(x) + \mathrm{i}\eta^{-1}_*(f(x)+\mathrm{i}f'(x)y)\chi'(y/\eta_*))\mathcal{E}(z) \mathrm{d}x\mathrm{d}y +O_{\prec}\biggl(\frac{N^{2\alpha -2} ||f''||_{1}}{ \sqrt{V_{ii}(f)}} \biggr) \\
     &= \frac{\mathrm{i}}{2\pi \sqrt{V_{ii}(f)}}\int_{\Omega_{\alpha}} (\mathrm{i}y\chi(y/\eta_*)f''(x) + \mathrm{i}\eta^{-1}_*(f(x)+\mathrm{i}f'(x)y)\chi'(y/\eta_*))\mathcal{E}(z) \mathrm{d}x\mathrm{d}y +O_{\prec}(N^{-\tau})
\end{align*}
for $\alpha < \tau/4$,  where
\[ \mathcal{E}(z) \deq \bb E[e(\lambda) (G_{ii}(z) - \bb EG_{ii}(z))]\,.\]
Now, define
\[ e_{\alpha}(\lambda) \deq \text{exp}\biggl[ \frac{\mathrm{i}\lambda}{2\pi \sqrt{V_{ii}(f)}}\int_{\Omega_{\alpha}}(\mathrm{i}y\chi(y/\eta_*)f''(x) + \mathrm{i}\eta^{-1}_*(f(x)+\mathrm{i}f'(x)y)\chi'(y/\eta_*))(G_{ii}-\bb EG_{ii})\mathrm{d}x\mathrm{d}y \biggr]\,.\]
Using \eqref{eq:e} and the fact that $|e^{\mathrm{i} B}-e^{\mathrm{i}C} |\leqslant |B-C|$ for $B, C \in \bb R$, we have
\begin{equation}
\label{er}
|e(\lambda) - e_{\alpha}(\lambda)|  \prec N^{-\tau}\,.
\end{equation}
The following lemma will be useful.
\begin{lemma} \label{lemma:integral_estimate}
    Let $W(z)$ be a holomorphic function on $\bb C\setminus \bb R$ satisfying 
\[ W(z) \prec | \textup{Im}(z)|^{-s} \]
for some $0\leqslant s<2$, then 
\[\int_{\Omega_{\alpha}}(\mathrm{i}y\chi(y/\eta_*)f''(x) + \mathrm{i}\eta^{-1}_*(f(x)+\mathrm{i}f'(x)y)\chi'(y/\eta_*))W\mathrm{d}x\mathrm{d}y \prec \eta^{1-s}_*+ (N^{1-\alpha})^{s-1}\,.\]
\end{lemma}
\begin{proof}
    Applying integration by parts, we have
\begin{align*} 
\bigg|\int_{\Omega_{\alpha}}y\chi(y/\eta_*)f''(x)W\mathrm{d}x\mathrm{d}y\biggr| &= \bigg|\int_{\Omega_{\alpha}}y\chi(y/\eta_*)f'(x)W'\mathrm{d}x\mathrm{d}y\biggr|\\&\prec \int_{\Omega_{\alpha}}|y^{-s}\chi(y/\eta_*)f'(x)|\mathrm{d}x\mathrm{d}y\\
&\prec \text{log}(N)((2\eta_*)^{1-s}+(N^{\alpha-1})^{1-s})||f'||_{1}\\
&\prec \eta^{1-s}_* + (N^{1-\alpha})^{s-1}\,, \numberthis \label{eq:integral_1}
\end{align*}
where we used $|W'|\prec |y|^{-s-1}$ by Cauchy's integral formula. Similarly,
\begin{align*}
\bigg|\int_{\Omega_{\alpha}} \mathrm{i}\eta^{-1}_*f(x)\chi'(y/\eta_*)W \mathrm{d}x\mathrm{d}y\bigg|
&\prec \text{log}(N)((2\eta_*)^{1-s}+\eta_*^{1-s})\eta^{-1}_*||f||_{1}\\
&\prec \eta^{1-s}_* \,,\numberthis \label{eq:integral_2}
\end{align*}
and
\begin{align*}
\bigg|\int_{\Omega_{\alpha}}\eta^{-1}_*f'(x)y\chi'(y/\eta_*)W \mathrm{d}x\mathrm{d}y\bigg| 
&\prec \eta^{-1}_*((2\eta_*)^{2-s}+\eta^{2-s}_*)||f'||_{1}\\
&\prec \eta^{1-s}_* \,.\numberthis \label{eq:integral_3}
\end{align*}
Adding \eqref{eq:integral_1}-\eqref{eq:integral_3} completes the proof.
\end{proof}

    Let $\mathcal{E}_{\alpha}(z)$ be the same as $\mathcal{E}(z)$ but with $e_{\alpha}(\lambda)$ in place of $e(\lambda)$. Applying Lemma \ref{lemma:integral_estimate} with $W=\mathcal{E}-\mathcal{E}_{\alpha} \prec N^{-\tau}(q^{-1}+(Ny)^{-1/2})$, we have

\begin{align*}\label{ee}
\biggl|\int_{\Omega_{\alpha}}(\mathrm{i}y\chi(y/\eta_*)f''(x) + \mathrm{i}\eta^{-1}_*(f(x)+\mathrm{i}f'(x)y)\chi'(y/\eta_*))(\mathcal{E}-\mathcal{E}_{\alpha})\mathrm{d}x\mathrm{d}y\biggr| &\prec N^{-\tau}(\eta_*/q+(\eta_*/N)^{1/2}) \\
&\asymp N^{-\tau}\sqrt{V_{ii}(f)}\,.
\end{align*}
Hence,
\[\psi_{ii}'(\lambda) = \frac{\mathrm{i}}{2\pi \sqrt{V_{ii}(f)}}\int_{\Omega_{\alpha}}(\mathrm{i}y\chi(y/\eta_*)f''(x) + \mathrm{i}\eta^{-1}_*(f(x)+\mathrm{i}f'(x)y)\chi'(y/\eta_*))\mathcal{E}_{\alpha}(z) \mathrm{d}x\mathrm{d}y +O_{\prec}(N^{-\tau})\,.\]
Using Lemma \ref{lemma:integral_estimate}, we could estimate the derivatives of $e_{\alpha}(\lambda)$ w.r.t. entries of $H$. We have, for $k\neq i$, that
\begin{align*}
    \frac{\partial e_{\alpha}(\lambda)}{\partial H_{ki}} &= e_{\alpha}(\lambda)\biggl[ \frac{-2\mathrm{i}\lambda}{\pi(1+\delta_{ki}) \sqrt{V_{ii}(f)}}\int_{\Omega_{\alpha}}\partial_{\bar{z}}\tilde{f}(z)G_{ii}G_{ik}\mathrm{d}^2z \biggr]\\
     &\prec 1 \,,\numberthis \label{e1}\\
    \sum_{k=1}^N \frac{\partial e_{\alpha}(\lambda)}{\partial H_{ki}} &= e_{\alpha}(\lambda)\biggl[ \frac{-2\mathrm{i}\lambda}{\pi(1+\delta_{ki})\sqrt{V_{ii}(f)}}\int_{\Omega_{\alpha}}\partial_{\bar{z}}\tilde{f}(z)G_{ii}\sum_{k=1}^N G_{ik}\mathrm{d}^2z \biggr] \\
    &\prec N^{1/2}\,, \numberthis \label{e2}\\
    \frac{\partial^2 e_{\alpha}(\lambda)}{\partial H_{ki}^2} &= \begin{aligned}[t]
        e_{\alpha}(\lambda)\biggl[ &\frac{-2\mathrm{i}\lambda}{\pi(1+\delta_{ki}) \sqrt{V_{ii}(f)}}\int_{\Omega_{\alpha}} \partial_{\bar{z}}\tilde{f}(z)G_{ii}G_{ik}\mathrm{d}^2z \biggr]^2 \\ &  + e_{\alpha}(\lambda)\frac{2\mathrm{i}\lambda}{\pi(1+\delta_{ki})^2 \sqrt{V_{ii}(f)}}\int_{\Omega_{\alpha}}\partial_{\bar{z}}\tilde{f}(z)(G_{ii}^2G_{kk}+3G_{ki}^2G_{ii})\mathrm{d}^2z
        \end{aligned}
    \\
    &= e_{\alpha}(\lambda)\frac{2\mathrm{i}\lambda}{\pi(1+\delta_{ki})^2 \sqrt{V_{ii}(f)}}\int_{\Omega_{\alpha}}\partial_{\bar{z}}\tilde{f}(z)m^3 \mathrm{d}^2z  + O_{\prec}(1)\,, \numberthis \label{e3}
\end{align*}
where we have applied Theorem \ref{thm:semiA} for \eqref{e1} and \eqref{e3}, and Theorem \ref{thm:isoA} for \eqref{e2}. In general, we have 
\begin{equation}
\label{ep}
\bigg|\frac{\partial^r e_{\alpha}(\lambda)}{\partial H_{ki}^r}\bigg| \prec \eta_*V_{ii}(f)^{-1/2}\,.
\end{equation}

Next, we would extract the leading-order terms by performing a cumulant expansion on $\mathcal{E}_{\alpha}$.
\begin{lemma}\label{lemma:E}
\begin{equation}
\label{eq:lemmastat}
\begin{aligned}(z+m)\mathcal{E}_{\alpha}(z) &= N^{-1} \bb E e_{\alpha}(\lambda)\biggl[ \frac{-2\mathrm{i}\lambda}{\pi \sqrt{V_{ii}(f)}}\int_{\Omega_{\alpha}}\partial_{\bar{z'}}\tilde{f}(z')G_{ii}(z')(G(z')G(z))_{ii}\mathrm{d}^2z' \biggr]\\&-\frac{\mathrm{i}\lambda N}{\pi\sqrt{V_{ii}(f)}} \mathcal {C}_{4}(H_{12}) \bb Ee_{\alpha}(\lambda)\biggl(m^2(z)\int_{\Omega_{\alpha}}\partial_{\bar{z'}}\tilde{f}(z')m^3(z') \mathrm{d}^2z' \biggr) \\&+ O_{\prec}\biggl( \frac{1}{q^2} + \frac{1}{q^2(y + |2-|x||)}+\frac{1}{Ny} + \frac{1}{\sqrt{N}}\biggr)
\end{aligned}
\end{equation}
uniformly in $z=x+\mathrm{i}y \in \Omega_{\alpha}$.
\end{lemma}
\begin{proof} By Lemma \ref{lemma:cumulant},
\begin{align*}
    z\mathcal{E}_{\alpha}(z) &= \bb E[e_{\alpha}(\lambda)z(G_{ii}(z) - \bb EG_{ii}(z))] \\
    &= \sum_{k=1}^N\bb E[e_{\alpha}(\lambda)\langle G_{ik}A_{ki}\rangle] \\
    &= \sum_{k=1}^N\bb E[\langle e_{\alpha}(\lambda) \rangle G_{ik}H_{ki}] + \sum_{k=1}^N\bb E[\langle e_{\alpha}(\lambda)\rangle G_{ik}f(ee^*)_{ki}]\\
    &=  N^{-1}\sum_{k=1}^N \bb E\langle e_{\alpha}(\lambda) \rangle\frac{\partial G_{ik}}{\partial H_{ki}}(1+\delta_{ki}) + N^{-1}\sum_{k=1}^N \bb E\frac{\partial \langle e_{\alpha}(\lambda) \rangle}{\partial H_{ki}}G_{ik}(1+\delta_{ki}) \\ & - N^{-1} \bb E\frac{\partial \langle e_{\alpha}(\lambda)\rangle G_{ii}}{\partial H_{ii}} +\sum_{r=2}^l
  L_{r} + \sum_{k=1}^N R_{l+1}^{(k)} + \sum_{k=1}^N\bb E[\langle e_{\alpha}(\lambda)\rangle G_{ik}f(ee^*)_{ki}] \\
 &\deq  (a) + (b) - N^{-1} \bb E\frac{\partial \langle e_{\alpha}(\lambda) \rangle G_{ii}}{\partial H_{ii}} +\sum_{r=2}^l
  L_{r} + \sum_{k=1}^N R_{l+1}^{(k)} + (c)\,,\numberthis \label{eq:cumulantE}
\end{align*}
where
 \[ L_r = \sum_{k=1}^N \frac{1}{r!} \mathcal{C}_{r+1}(H_{ki}) \bb E\biggl(\frac{\partial^r \langle e_{\alpha}(\lambda)\rangle G_{ik}}{\partial H_{ki}^r} \biggr)\,.
\]
By Theorem \ref{thm:isoA} and the fact that $f\asymp q$, we have
    \begin{align*}
        (c) = \sum_{k=1}^N\bb E[\langle e_{\alpha}(\lambda)\rangle G_{ik}f(ee^*)_{ki}] &= \frac{f}{N}\bb E\biggl[\langle e_{\alpha}(\lambda)\rangle \sum_{k=1}^N G_{ik}\biggr]\\
        &\prec N^{-1/2}\,. \numberthis \label{eq:c1}
    \end{align*}
  By the differential rule \eqref{eq:H_diff}, Theorem \ref{thm:semiA} and Corollary \ref{corollary:diff}
\begin{align*}
    (a) &= N^{-1} \sum_{k=1}^{N} \bb E\langle e_{\alpha}(\lambda) \rangle(-G_{ii}G_{kk}-G_{ik}G_{ki}) \\
    &= -N^{-1} \bb E \langle e_{\alpha}(\lambda) \rangle(G_{ii} \text{Tr}G+(G^2)_{ii}) \\
    &=-\bb E \langle e_{\alpha}(\lambda)\rangle   (\underline{G}G_{ii}-G_{ii} m+G_{ii} m)  -N^{-1} \bb Ee_{\alpha}(\lambda) \langle (G^2)_{ii} - m' \rangle\\
    &=-\bb E e_{\alpha}(\lambda)\langle G_{ii} \rangle m+O_{\prec}\biggl( \frac{1}{q^2}+\frac{1}{q^2(y + |2-|x||)}+\frac{1}{Ny} +(Ny)^{-1}\biggl(\frac{1}{q}+\frac{1}{\sqrt{Ny}}\biggr)\biggr)\,. \numberthis \label{eq:cumulanta} \\
    \end{align*}
Similarly,
    \begin{align*}
        (b) &= N^{-1}\sum_{k=1}^N \bb E\frac{\partial \langle e_{\alpha}(\lambda) \rangle}{\partial H_{ki}}G_{ik}(1+\delta_{ki}) \\
        &= N^{-1}\sum_{k=1}^N \bb E e_{\alpha}(\lambda)\biggl[ \frac{-2\mathrm{i}\lambda}{\pi \sqrt{V_{ii}(f)}}\int_{\Omega_{\alpha}}\partial_{\bar{z'}}\tilde{f}(z')G_{ii}(z')G_{ik}(z')G_{ik}(z)\mathrm{d}^2z' \biggr]\\
        &=N^{-1} \bb E e_{\alpha}(\lambda)\biggl[ \frac{-2\mathrm{i}\lambda}{\pi \sqrt{V_{ii}(f)}}\int_{\Omega_{\alpha}}\partial_{\bar{z'}}\tilde{f}(z')G_{ii}(z')(G(z')G(z))_{ii}\mathrm{d}^2z'\biggr]\,.  \numberthis \label{eq:cumulantb}
    \end{align*}
$L_2$ consists of terms of three different types, depending on how many derivatives acts on $G_{ik}$. We now show that they are all non-leading terms. Using \eqref{e1} - \eqref{e3}, together with Theorem \ref{thm:semiA} and Theorem \ref{thm:isoA}, we have that for some constants $a_{1}, a_2$ depending on $\delta_{ki}$
\begin{align*}
   L_{2,1} &= \sum_{k=1}^N \frac{1}{2} \mathcal{C}_{3}(H_{ki}) \bb E \biggl(\langle e(\lambda)\rangle \frac{\partial^2 G_{ik}}{\partial H_{ki}^2} \biggr)\\
   &= O((Nq)^{-1})\sum_{k=1}^N \bb E \langle e(\lambda)\rangle \biggl(a_1G_{ii}G_{ik}G_{kk} + a_{2}G_{ik}G_{ki}G_{ki}\biggr)\\
   &= O((Nq)^{-1})\sum_{k=1}^N \bb E m^2e(\lambda) a_1 \langle 
 G_{ik} \rangle + O_{\prec}\biggl( q^{-1}\biggl(\frac{1}{q}+\frac{1}{\sqrt{Ny}}\biggr)^2\biggr)\\
 &\prec  N^{-1/2}q^{-2}+q^{-1}\biggl(\frac{1}{q}+\frac{1}{\sqrt{Ny}}\biggr)^2 \,,\numberthis \label{eq:cumulantl21}\\
     L_{2,2} &= \sum_{k=1}^N  \mathcal{C}_{3}(H_{ki}) \bb E \biggl(\frac{\partial\langle e_{\alpha}(\lambda)\rangle}{\partial H_{ki}} \frac{\partial G_{ik}}{\partial H_{ki}} \biggr)\\
   &= O((Nq)^{-1})\sum_{k=1}^N \bb E (1+\delta_{ki})^{-1}\frac{\partial e_{\alpha}(\lambda)}{\partial H_{ki}}(G_{ii}G_{kk}+G_{ik}G_{ki})\\
    &= O((Nq)^{-1})\sum_{k=1}^N \bb E  \frac{\partial e_{\alpha}(\lambda)}{\partial H_{ki}}G_{ii}G_{kk} +O_{\prec}\biggl(q^{-1}\biggl(\frac{1}{q}+\frac{1}{\sqrt{Ny}}\biggr)^2\biggr) \\
    &= O((Nq)^{-1}) m^2\sum_{k=1}^N\bb E\frac{\partial e_{\alpha}(\lambda)}{\partial H_{ki}} +O_{\prec}\biggl(q^{-1}\biggl(\frac{1}{q}+\frac{1}{\sqrt{Ny}}\biggr)  \biggr)\\
    &\prec q^{-1}\biggl(\frac{1}{q}+\frac{1}{\sqrt{Ny}}\biggr) \,,\numberthis \label{eq:cumulantl22}\\
    L_{2,3} &=  \sum_{k=1}^N \frac{1}{2} \mathcal{C}_{3}(H_{ki}) \bb E \biggl(\frac{\partial^2\langle e_{\alpha}(\lambda)\rangle}{\partial H_{ki}^2} G_{ik}\biggr)\\
        &=\begin{aligned}[t]
            \frac{1}{2}\sum_{k=1}^N 
 \mathcal{C}_{3}(H_{ki}) \bb E \biggl( e_{\alpha}(\lambda)\frac{2\mathrm{i}\lambda}{\pi(1+\delta_{ki})^2 \sqrt{V_{ii}(f)}}\int_{\Omega_{\alpha}}\partial_{\bar{z'}}\tilde{f}(z')&m^3(z') \mathrm{d}^2z' G_{ik}\biggr)\\&+O_{\prec}\biggl(q^{-1}\biggl(\frac{1}{q}+\frac{1}{\sqrt{Ny}}\biggr)\biggr)
 \end{aligned}\\
 &\prec  \frac{\eta_*}{q^{2}N^{1/2}\sqrt{V_{ii}(f)}}+q^{-1}\biggl(\frac{1}{q}+\frac{1}{\sqrt{Ny}}\biggr) \\
 &\prec q^{-1}\biggl(\frac{1}{q}+\frac{1}{\sqrt{Ny}}\biggr)\,.\numberthis \label{eq:cumulantl23}
\end{align*}
Similarly for $L_{3}$,
\begin{align*}
   L_{3,1} &= \sum_{k=1}^N \frac{1}{6} \mathcal{C}_{4}(H_{ki}) \bb E \biggl(\langle e(\lambda)\rangle \frac{\partial^3 G_{ik}}{\partial H_{ki}^3} \biggr)\\
   &\prec q^{-2}\,, \numberthis \label{eq:cumulantl31}\\
     L_{3,2} &= \sum_{k=1}^N \frac{1}{2} \mathcal{C}_{4}(H_{ki}) \bb E \biggl(\frac{\partial\langle e_{\alpha}(\lambda)\rangle}{\partial H_{ki}} \frac{\partial^2 G_{ik}}{\partial H_{ki}^2} \biggr)\\
   &\prec q^{-2}\,, \numberthis \label{eq:cumulantl32}\\
    L_{3,4} &=  \sum_{k=1}^N\frac{1}{6}  \mathcal{C}_{4}(H_{ki}) \bb E \biggl(\frac{\partial^3\langle e_{\alpha}(\lambda)\rangle}{\partial H_{ki}^3} G_{ik}\biggr)\\
    &\prec q^{-2}\eta_*V_{ii}(f)^{-1/2} \biggl(\frac{1}{q} + \frac{1}{\sqrt{Ny}} \biggr) \\
    &\prec q^{-1}\biggl(\frac{1}{q} + \frac{1}{\sqrt{Ny}} \biggr)\,.
        \numberthis \label{eq:cumulantl34}
\end{align*}
    The remaining leading order term comes from $L_{3,3}$. Using \eqref{e3} and Theorem \ref{thm:semiA}
\begin{align*}
     L_{3,3} &= \frac{1}{2}\sum_{k=1}^N  \mathcal{C}_{4}(H_{ki}) \bb E \biggl(\frac{\partial^2\langle e_{\alpha}(\lambda)\rangle}{\partial H_{ki}^2} \frac{\partial G_{ik}}{\partial H_{ki}} \biggr) \\
     &= \frac{-1}{2}\sum_{k=1}^N(1+\delta_{ik})^{-1}  \mathcal {C}_{4}(H_{ki}) \bb E \biggl(\frac{\partial^2 e_{\alpha}(\lambda)}{\partial H_{ki}^2} (G_{ii}G_{kk}+G_{ik}^2) \biggr) \\
     \begin{split}
         &=\frac{-N}{2} \mathcal {C}_{4}(H_{12}) \bb E \biggl(m^2(z)e_{\alpha}(\lambda)\frac{2\mathrm{i}\lambda}{\pi \sqrt{V_{ii}(f)}}\int_{\Omega_{\alpha}}\partial_{\bar{z'}}\tilde{f}(z')m^3(z') \mathrm{d}^2z' \biggr) \\ & \qquad \qquad \qquad \qquad \qquad \qquad \qquad \qquad + O_{\prec}\biggl(q^{-1}\biggl(\frac{1}{q} + \frac{1}{\sqrt{Ny}}\biggr)\biggr) \,.
         \end{split} \numberthis \label{eq:cumulantl33}
\end{align*}
 Using \eqref{eq:H_diff}, \eqref{ep} and Theorem \ref{thm:semiA}, we have that for $r \geqslant 4$,
\begin{align*}
L_{r}  &= \sum_{k=1}^N \frac{1}{r!} \mathcal{C}_{r+1}(H_{ki}) \bb E\biggl(\frac{\partial^r \langle e_{\alpha}(\lambda)\rangle G_{ik}}{\partial H_{ki}^r} \biggr) \\&\prec \eta_*q^{-3}V_{ii}(f)^{-1/2} \\
& \prec q^{-2}\,. \numberthis \label{eq:generalr}
\end{align*}
Also, using \eqref{eq:H_diff} and Lemma \ref{lemma:integral_estimate},
\begin{align*}
N^{-1} \bb E\frac{\partial \langle e_{\alpha}(\lambda) \rangle G_{ii}}{\partial H_{ii}} &= N^{-1} \bb E\frac{\partial  e_{\alpha}(\lambda)}{\partial H_{ii}}G_{ii}-N^{-1} \bb E \langle e_{\alpha} (\lambda)\rangle G_{ii}G_{ii}\\
&\prec N^{-1}\eta_*V_{ii}(f)^{-1/2}\\
&\prec q/N \,.\numberthis \label{eq:K1}
\end{align*}
 Substituting \eqref{eq:c1}-\eqref{eq:K1} into \eqref{eq:cumulantE} and letting $l$ be large enough so that the remainder term is negligible completes the proof.
\end{proof}

    Let $z' = x'+\mathrm{i}y'$. We now simplify the first integral on the R.H.S of Lemma \ref{lemma:E}, note that by Theorem \ref{thm:semiA}, \[(G_{ii}(z')-m(z'))G_{ik}(z') \prec \frac{1}{q^2}+\frac{1}{Ny'}\] 
    for $i\ne k$. Applying Lemma \ref{lemma:integral_estimate}, we have,
\begin{align*}
\bigg|\int_{\Omega_{\alpha}}\partial_{\bar{z'}}\tilde{f}(z')(G_{ii}(z')-m(z'))G_{ik}(z')\mathrm{d}^2z'\bigg| 
&\prec \eta_*q^{-2}+N^{-1}\,.
\numberthis \label{eq:T_a1}
\end{align*}
Hence,
\begin{equation}
\begin{split}
    \label{eq:inter}
    (b) = N^{-1} \bb E e_{\alpha}(\lambda)\biggl[ \frac{-2\mathrm{i}\lambda}{\pi \sqrt{V_{ii}(f)}}&\int_{\Omega_{\alpha}}\partial_{\bar{z'}}\tilde{f}(z')m(z')(G(z')G(z))_{ii}\mathrm{d}^2z' \biggr] \\& + O_{\prec}\biggl( \biggl(\frac{1}{q}+\sqrt{\frac{1}{\eta_*N}} \biggr)\biggl(\frac{1}{q}+\frac{1}{\sqrt{Ny}}\biggr)\biggr)\,.
\end{split}
\end{equation}

We now estimate the error of further replacing $G$ in \eqref{eq:inter} by $m$. Note that
\begin{align*}
    (G(z')G(z))_{ii} &= \frac{G_{ii}(z')-G_{ii}(z)}{z'-z}\\
    &= \frac{m(z')-m(z)}{z'-z} + \frac{G_{ii}(z')-m(z')-(G_{ii}(z)-m(z))}{z'-z}\,. \numberthis \label{eq:GG}
\end{align*}
    When $z$ and $z'$ are in different half-planes, then the second term in \eqref{eq:GG} is of $O_{\prec}((q^{-1}+(Ny)^{-1/2}+(Ny')^{-1/2})(y+y')^{-1})$. When $z$ and $z'$ are in the same half-plane, we split into the cases $y<2y'$ and $y>2y'$, in the former case, let $\gamma$ be a straight line path from $z$ to $z'$. Then
    
\begin{align*}
    \bigg|\frac{G_{ii}(z')-m(z')-(G_{ii}(z)-m(z))}{z'-z}\bigg| &= \bigg|\frac{1}{z'-z} \int_{\gamma}(G_{ii}-m)'\mathrm{d}s\bigg|\\
    &\prec \sup_{z\in \gamma}|(G_{ii}-m)'| \\
    &\prec y^{-1}\biggl(\frac{1}{q}+\frac{1}{\sqrt{Ny}}\biggr) 
\end{align*}
    for $y > 2y'$, we have $|z'-z| \geqslant y/2$ and
\begin{align*}
    \bigg|\frac{G_{ii}(z')-m(z')-(G_{ii}(z)-m(z))}{z'-z}\bigg| &\prec y^{-1}\biggl(\frac{1}{q}+\frac{1}{\sqrt{Ny}}+\frac{1}{\sqrt{Ny'}}\biggr)\\
    &\prec y^{-1}\biggl(\frac{1}{q}+\frac{1}{\sqrt{Ny'}}\biggr)\,.
\end{align*}
Combining the above,
\[\bigg|\frac{G_{ii}(z')-m(z')-(G_{ii}(z)-m(z))}{z'-z}\bigg| \prec y^{-1}\biggl(\frac{1}{q}+\frac{1}{\sqrt{Ny'}}+\frac{1}{\sqrt{Ny}} \biggr)\,.\]
Applying Lemma \ref{lemma:integral_estimate}, we have 
\begin{align*}
(b) = N^{-1} \bb E e_{\alpha}(\lambda)\biggl[ \frac{-2\mathrm{i}\lambda}{\pi \sqrt{V_{ii}(f)}}&\int_{\Omega_{\alpha}}\partial_{\bar{z'}}\tilde{f}(z')m(z')\frac{m(z')-m(z)}{z'-z}\mathrm{d}^2z' \biggr] \\& + O_{\prec}\biggl( \biggl(\frac{1}{q}+\sqrt{\frac{1}{\eta_*N}} \biggr)\biggl(\frac{1}{q}+\frac{1}{\sqrt{Ny}}\biggr) +\frac{1}{Ny}+\frac{\sqrt{\eta_*}}{Ny^{3/2}}\biggr) \,.\numberthis \label{eq:finalb}
\end{align*}
Note that from \eqref{eq:m}, we have $(z+m)^{-1}=-m$. Substituting \eqref{eq:finalb} into \eqref{eq:lemmastat} and then applying Lemma \ref{lemma:integral_estimate}, we have
\begin{align*}\psi_{ii}'(\lambda) &= -\lambda\bb E[e_{\alpha}(\lambda)] \frac{1}{\pi^2V_{ii}(f)}\int_{\Omega_{\alpha}^2}\partial_{\bar{z}}\tilde{f}(z)\partial_{\bar{z'}}\tilde{f}(z') \\
&\times \biggl[2N^{-1} m(z)m(z')\frac{m(z')-m(z)}{z'-z}+N\mathcal {C}_{4}(H_{12})m^3(z)m^3(z')\biggr]\mathrm{d}^2z'\mathrm{d}^2z  \\
& +O_{\prec}(N^{-\tau/100} + N^{-\alpha/2})\,.
\end{align*}

Let $\alpha = \tau/100$. Using the fact that $|m(z)| \asymp 1$ and $|m'(z)| \prec |y|^{-1/2}$, the integrand above is of the order $O_{\prec}(N^{-1}(|y|^{-1}+|y'|^{-1}) + q^{-2})$. Applying Lemma \ref{lemma:integral_estimate} twice,
\begin{align*}
    &\biggl|\int_{ \Omega_{\alpha}^2}\partial_{\bar{z}}\tilde{f}(z)\partial_{\bar{z'}}\tilde{f}(z') \biggl[2N^{-1} m(z)m(z')\frac{m(z')-m(z)}{z'-z}+N\mathcal {C}_{4}(H_{12})m^3(z)m^3(z')\biggr]\mathrm{d}^2z'\mathrm{d}^2z \biggr|\\
    &\prec \eta_*/N + \eta^2_*/q^2 \\
    &\prec V_{ii}(f)\,.
\end{align*}
Hence we could replace $e_{\alpha}(\lambda)$ by $e(\lambda)$ by \eqref{er}. Next, we argue that we could replace $\Omega_{\alpha}^2$ by $\bb C^2$ in the above integral. The contributions from the regions $(\bb C\setminus\Omega_{\alpha})^2$ and $(\bb C\setminus\Omega_{\alpha})\times\Omega_{\alpha}$ could be shown to be of order $O_{\prec}(N^{-\tau/100}V_{ii}(f))$ as follows:
\begin{align*}
    &\biggl|\int_{(\bb C \setminus \Omega_{\alpha})^2}\partial_{\bar{z}}\tilde{f}(z)\partial_{\bar{z'}}\tilde{f}(z') \biggl[2N^{-1} m(z)m(z')\frac{m(z')-m(z)}{z'-z}+N\mathcal {C}_{4}(H_{12})m^3(z)m^3(z')\biggr]\mathrm{d}^2z'\mathrm{d}^2z \biggr|\\
    &\prec \int_{(\bb C \setminus \Omega_{\alpha})^2}|\partial_{\bar{z}}\tilde{f}(z)\partial_{\bar{z'}}\tilde{f}(z')| \biggl[N^{-1}(|y|^{-1} + |y'|^{-1})+q^{-2}\biggr]\mathrm{d}^2z'\mathrm{d}^2z \\
    &\prec \int_{\bb C \setminus \Omega_{\alpha}} |\partial_{\bar{z}}\tilde{f}(z)|\biggl[N^{-1}(|y|^{-1}(N^{\alpha -1})^2||f''||_{1} + N^{\alpha -1 }||f''||_{1})+q^{-2}(N^{\alpha -1})^2||f''||_{1}\biggr]\mathrm{d}^2z\\
    &\prec N^{-1}(N^{\alpha-1})^{3}||f''||_{1}^2 +q^{-2}(N^{\alpha -1})^4||f''||_{1}^2 \\
    &\prec N^{-\tau/100}V_{ii}(f)\,.
\end{align*}
Similarly,
\begin{align*}
    &\biggl|\int_{(\bb C \setminus \Omega_{\alpha})\times\Omega_{\alpha}}\partial_{\bar{z}}\tilde{f}(z)\partial_{\bar{z'}}\tilde{f}(z') \biggl[2N^{-1} m(z)m(z')\frac{m(z')-m(z)}{z'-z}+N\mathcal {C}_{4}(H_{12})m^3(z)m^3(z')\biggr]\mathrm{d}^2z'\mathrm{d}^2z \biggr|\\
    &\prec N^{-1}(N^{\alpha - 1}\eta_*||f''||_{1} + (N^{\alpha -1})^2||f''||_{1})+q^{-2}(N^{\alpha -1})^2\eta_*||f''||_{1} \\
    &\prec N^{-\tau/100}V_{ii}(f)\,,
\end{align*}
where Lemma \ref{lemma:integral_estimate} is applied once for the second estimate.
This completes the proof of Proposition \ref{prop:char}.
\subsection{Evaluation of Integrals} \label{subsection}
We now evaluate the integrals appearing in Proposition \ref{prop:char} by applying Green's Theorem, which states that for any sufficiently smooth function $F(z)$,
\begin{equation}
\label{eq:green}
\int_{\Omega}\partial_{\bar{z}}F(z)\mathrm{d}^2z=\frac{-\mathrm{i}}{2}\int_{\partial \Omega}F(z)\mathrm{d}z\,.
\end{equation}
    Now let $\eta >0$ be small and define 
\[ \widetilde{\Omega}_{\eta} \deq \{ (x+\mathrm{i}y, x'+\mathrm{i}y') \in \bb C^2:|y| \geqslant \eta/2, |y'| \geqslant \eta/2\}\,.\]
The boundary $\partial\widetilde{\Omega}_{\eta}$ consists of four branches, corresponding to the possible sign combinations of the imaginary part of $z$ and $z'$. Consider first the branch $z' = x' + \mathrm{i}\eta/2$, $z= x-\mathrm{i}\eta/2$, we have
\begin{align*}
    &\quad\frac{m(z')-m(z)}{z'-z}\\ &= \frac{-\mathrm{i}\eta+x-x'+\sqrt{z'^2 -4}+\sqrt{z^2-4}}{2(x'-x+\mathrm{i}\eta)} \\
    &=\frac{(\sqrt{4-x'^2}+\sqrt{4-x^2})\eta}{2((x'-x)^2+\eta^2)} +\frac{O(\eta^2)}{(x'-x)^2+\eta^2}
    +\frac{(x'-x)(x-x'+\mathrm{i}\sqrt{4-x'^2}+\mathrm{i}\sqrt{4-x^2})}{2((x'-x)^2+\eta^2)}\,,
    \numberthis \label{1}\\ \\
    &\quad m(z')m(z)\\ &= \frac{(-x'+\mathrm{i}\sqrt{4-x'^2} + O(\eta))(-x-\mathrm{i}\sqrt{4-x^2} + O(\eta))}{4}\,.\numberthis \label{2}
\end{align*}
    As $\eta$ tends to $0$, multiplying \eqref{1} and \eqref{2} together and summing over analogous expressions from the other boundary branches yields $-2\pi\sqrt{4-x^2}\delta(x'-x)+\sqrt{4-x^2}\sqrt{4-x'^2}$. Using \eqref{eq:green},
\begin{align*}
     &\quad\int_{\bb C^2}\partial_{\bar{z}}\tilde{f}(z)\partial_{\bar{z'}}\tilde{f}(z') m(z)m(z')\frac{m(z')-m(z)}{z'-z}\mathrm{d}^2z'\mathrm{d}^2z \\ &= \lim_{\eta \to0^+} \int_{\widetilde{\Omega}_{\eta}}\partial_{\bar{z}}\tilde{f}(z)\partial_{\bar{z'}}\tilde{f}(z') m(z)m(z')\frac{m(z')-m(z)}{z'-z}\mathrm{d}^2z'\mathrm{d}^2z \\
     &= \frac{1}{2}\int_{-2}^2\int_{-2}^2f(x)f(x')\pi\sqrt{4-x^2}\delta(x'-x)\mathrm{d}x\mathrm{d}x' -\biggl( \frac{1}{2} \int_{-2}^2 f(x)\sqrt{4-x^2}\mathrm{d}x\biggr)^2\\
     &=\pi^2\int_{-2}^2 f(x)^2\varrho_{sc}(x)\mathrm{d}x -\biggl( \pi \int_{-2}^2 f(x)\varrho_{sc}(x)\mathrm{d}x\biggr)^2 \,.\numberthis \label{eq:integral1}
\end{align*}
Similarly, denoting $z= x+\mathrm{i}\eta/2$, $z' = x' + \mathrm{i}\eta/2$, we have
\begin{align*}
    &\quad m^3(z)m^3(z') - m^3(\bar{z})m^3(z')-m^3(z)m^3(\bar{z'})+m^3(\bar{z})m^3(\bar{z'})\\
    &=(m^3(z)-m^3(\bar{z}))(m^3(z')-m^3(\bar{z'}))\\
    &=-\sqrt{4-x^2}(1-x^2)\sqrt{4-x'^2}(1-x'^2) + O(\eta)\,,
\end{align*}
and hence
\begin{align*}
    \int_{\bb C^2}\partial_{\bar{z}}\tilde{f}(z)\partial_{\bar{z'}}\tilde{f}(z') m^3(z')m^3(z)\mathrm{d}^2z'\mathrm{d}^2z &= \biggl(\pi\int_{-2}^2 f(x)(1-x^2)\varrho_{sc}(x)\mathrm{d}x \biggr)^2 \,.\numberthis \label{eq:integral2}
\end{align*}
Substituting \eqref{eq:integral1} and \eqref{eq:integral2} into \eqref{eq:char}, then integrating w.r.t. $\lambda$ yields Theorem \ref{thm:1}.

To calculate the expectation, we apply Lemma \ref{lemma:cumu} and Lemma \ref{lemma:integral_estimate}, obtaining
\begin{align*}
    &\bb E[f_{ii}(A)]\\ &= \frac{1}{\pi} \int_{\bb C} \partial_{\bar{z}}f(z)\bb EG_{ii} \mathrm{d}^2z \\
    &=\begin{aligned}[t]
    \frac{1}{\pi} \int_{\bb C} \partial_{\bar{z}}f(z)m(z) \mathrm{d}^2z &+ \frac{1}{\pi}\int_{\Omega_{\tau/100}}\partial_{\bar{z}}f(z)(\bb E G_{ii} -m(z))\mathrm{d}^2z \\ &+\frac{1}{\pi}\int_{\bb C \setminus \Omega_{\tau/100}}\partial_{\bar{z}}f(z)(\bb E G_{ii} -m(z))\mathrm{d}^2z 
    \end{aligned}\\
    &= \int_{-2}^2 f(x)\varrho_{sc}(x)\mathrm{d}x + O_{\prec}\biggl(N^{-\tau/100}(\eta_*/q + (\eta_*/N)^{1/2}) + ||f''||_{1}N^{\tau/10 -2} \biggr) \\
    &= \int_{-2}^2 f(x)\varrho_{sc}(x)\mathrm{d}x + O_{\prec}(N^{-\tau/100}V_{ii}(f)^{1/2})\,.
\end{align*}
{\small
	
	\bibliography{bibliography} 
	
	\bibliographystyle{amsplain}
    }
\end{document}